\documentclass[twoside,reqno, 10pt]{amsart}

 \usepackage{amsmath, amssymb, amsfonts}

\newcommand{\fdiff}[4]{\ensuremath{ \upsilon^{#4 }_{ #3 }  #1  \left(  #2 \right)  }}
\newcommand{\fdiffplus}[3]{ \fdiff {#1}{#2}{+}{#3} }
\newcommand{\fdiffmin}[3]{ \fdiff {#1}{#2}{-}{#3} }
\newcommand{\fdiffpm}[3]{ \fdiff {#1}{#2}{\pm}{#3} }
\newcommand{\fracvar}[4]{\ensuremath{ \upsilon_{ #3 }^{ #4} \left[ #1 \right] \left(  #2 \right)   }}
\newcommand{\fracvarplus}[3]{ \fracvar {#1}{#2}{\epsilon+}{#3} }
\newcommand{\fracvarmin}[3]{ \fracvar {#1}{#2}{\epsilon -}{#3} }

\newcommand{\fracvarpm}[3]{ \fracvar {#1}{#2}{#3}{\epsilon \pm} }

\newcommand{\svar}[4]{\ensuremath{ \mathcal{S}_{ #3 }^{ #4} \left[ #1 \right] \left(  #2 \right)   }}
\newcommand{\svarplus}[3]{ \svar{#1}{#2}{#3}{\epsilon+} }
\newcommand{\svarmin}[3]{ \svar{#1}{#2}{#3}{\epsilon -} }
\newcommand{\svarpm}[3]{ \svar{#1}{#2}{#3}{\epsilon \pm} }

\newcommand{\llim}[3]{\ensuremath{ \lim\limits_{ #1 \rightarrow #2} #3 }}
\newcommand{\fclass}[2]{\ensuremath{  \mathbb{#1}^{\, #2} }}

\newcommand{\holder}[1]{\fclass{H}{#1} }

\newcommand{\epnt}{\; .}
\newcommand{\ecma}{\; ,}

\newcommand{\soc}[2]{\ensuremath{ \chi_{#1}^{#2} }}
\newcommand{\bigoh}[1]{ \ensuremath{ \mathcal{O} \left(  #1 \right) }   }

\newcommand{\frdiffi}[3]{\ensuremath{\,_{#3}\mathbf{D}^{#1}_{#2}}}

\newcommand{\frdiffix}[2]{\frdiffi{#1}{ #2}{ }}

\newcommand{\frdiffii}[3]{\ensuremath{\,_{#3}\mathbf{I}^{#1}_{#2}}}

\newcommand{\frdiffiix}[2]{\frdiffii{#1}{x}{#2}}

\newcommand{\frdiffir}[3]{\ensuremath{\,_{#3 RL}\mathbf{D}^{#1}_{#2}}}

\newcommand{\frdiffixr}[2]{\frdiffir{#1}{#2}{ }}

\newtheorem{theorem}{Theorem}
\newtheorem{lemma}{Lemma}
\newtheorem{corollary}{Corollary}
\newtheorem{proposition}{Proposition}

\newtheorem{definition}{Definition}
\newtheorem{remark}{Remark}
\newtheorem{example}{Example}

 \usepackage{graphicx, color, amssymb}
 
 \usepackage{lineno}
\begin{document}

   \title[Scale space analysis of  singular functions]{Characterization of strongly  non-linear and singular functions by scale space analysis}
        \author {Dimiter Prodanov}
      \address{Correspondence: Environment, Health and Safety, IMEC vzw, Kapeldreef 75, 3001 Leuven, Belgium
      	e-mail: Dimiter.Prodanov@imec.be, dimiterpp@gmail.com}
    	
   	\begin{abstract}
    A central notion of physics is the rate of change. 
    While mathematically the concept of derivative represents an idealization of the linear growth, 
    power law types of non-linearities even in noiseless physical signals cause derivative divergence.
    As a way to characterize change of strongly nonlinear signals, this work introduces the concepts of scale space embedding and scale-space velocity operators.	Parallels with the scale relativity theory and fractional calculus are discussed. 
    The approach is exemplified by an application to  De Rham's function.
   	It is demonstrated how scale space embedding presents a simple way of characterizing the growth 
   	of functions defined by means of iterative function systems.
  
    \end{abstract}

\maketitle
 

   	
\section{Introduction}
\label{seq:intro}

A central notion of physics is the rate of change. 
It can be argued that this perception inspired Newton and Leibniz to develop the apparatus of differential calculus. 
In its initial treatment, the notion of \textit{derivative} is limited only to the idealization of linear rate of change. 
 
On the other hand, modern calculus is not limited only to to using only ordinary derivatives and linear growth.
For example, \textit{fractional calculus} is used to model strongly non-linear phenomena, such as the those governed by power laws, or phenomena resulting in features at all scales -- i.e. fractals.
Many types of anomalous transport phenomena in physics are modeled using fractional diffusion-like equations \cite{Mainardi2007}.
The spatial complexity of a medium can impose geometrical constraints on transport processes on all length scales that can fundamentally alter the laws of standard diffusion \cite{Metzler2004}. 
Interesting links between fractional calculus and fractal random walks and statistical physics has been studied by \cite{Hilfer1995}.
For example, Continuous-Time-Random-Walks models are used to model anomalous diffusion phenomena \cite{Henry2006}. 
These models in limit yield fractional Fokker-Planck equations in terms of Caputo fractional derivatives \cite{Metzler2004}. 
Fractional diffusion models have been employed, for example, in hydrology \cite{Meerschaert2006}; heat conduction \cite{Hristov2011}; and anomalous drug absorption and disposition processes  \cite{Dokoumetzidis2009}.

It should be noted that in the general case fractional differential equations are difficult to solve exactly even with special functions.
A fractional embedding of an ordinary differential equation is not unique and different embeddings correspond to different microscopic physical models \cite{Hilfer2011}.
It also introduces sensitivity of the boundary conditions, which requires regularization procedures leading  to Caputo-type of derivatives.    
Another disadvantage of the fractional calculus is the difficulty to interpret the fractional derivative or integral because of their \textit{non-local character}. 
All this led to development of some alternatives of fractional calculus, retaining of the locality of the description \cite{Cherbit1991}. 
For example, such are the calculus on time-scales and local fractional calculus \cite{Hilger1998, Kolwankar1997a}.

Classically, physical variables, such as velocity or acceleration, are considered to be differentiable functions of position. 
On the other hand, quantum mechanical paths were found to be non-differentiable and stochastic in simulations \cite{Amir-Azizi1987}. 
The relaxation of the differentiability assumption could open new avenues in describing physical phenomena, for example, using the \textit{scale relativity theory}  developed by  Nottale \cite{Nottale1989}, which assumes  fractality of geodesics and quantum-mechanical paths.  
The main tenet of the scale relativity theory promotes the scale (temporal and/or spatial) of observation as an independent physical variable. 
This scale is considered to be irreducible in the sense that it can not be removed from the equations of dynamics by means of a mathematical limiting process.  
Moreover, by endowing it with the status of an independent variable the theory admits a continuum representation of the scale as a quantity.

In the present paper I follow this direction of investigation and give a more precise meaning to the physical intuition offered by scale relativity
by introducing the formal concept of a \textit{scale space} and defining mappings between derivatives in "physical" space and certain functions in this abstract scale space.
The paper demonstrates in an indirect manner also a connection between the principle of scale relativity and Caputo fractional derivatives. 
Idealized physical signals considered in this work are identified with H\"older-continuous functions (see \ref{sec:holder}).

\subsection{Fractional velocity}
\label{sec:fracvel}

Cherbit \cite{Cherbit1991} introduced the notion of \textsl{fractional velocity} as the limit of the fractional difference quotient. 
Its main application was the study of fractal phenomena and physical processes for which the instantaneous velocity is not well defined. 
\begin{definition}
	\label{def:fracvarn}
	Define \textsl{fractal variation} operators of  mixed order $n+\beta$ acting on a function $f(x)$ as
	\begin{align}
	\label{eq:fracdiffn}
	\fracvarplus {f}{x}{n+\beta} &:= (n+1)! \  \frac{ f( x+ \epsilon) - T_n (x, \epsilon) }{\epsilon^{n+\beta}}   \ecma   	\\
	\fracvarmin {f}{x}{n+\beta} &:= (-1)^{n} (n+1)! \ \frac{T_n (x, - \epsilon) - f (x -\epsilon)  }{\epsilon^{n+\beta}}   \epnt
	\end{align}
	where  $T_n(x, \epsilon), \ n \in \fclass{N}{}_0$ is the usual Taylor polynomial 
	\[ 
	T_n(x, \epsilon)= f(x) + \sum\limits_{k=1}^{n} \frac{f^{(k)}(x)}{k!} \epsilon^k \ecma
	\]
	and $\epsilon >0$ and  $0 < \beta \leq 1 $ are real parameters.
	\footnote{Positions of indices in the notation are switched compared to previous works \cite{Prodanov2015, Prodanov2016}.}
	Zero will be skipped from the notation in the purely fractional case.
\end{definition}

\begin{definition}[Fractional velocity]
	\label{def:frdiff2}
	Define the \textsl{fractional velocities} of order $n+ \beta$ of the  $f(x) $  as the limits (if they exist):
	\begin{align}
	\label{eq:fracdiff1}
	\fdiffplus {f}{x}{n+\beta} &:=  \llim{\epsilon}{0}{  \fracvarplus {f}{x}{n+\beta}}   \ecma   	\\
	\fdiffmin {f}{x}{n+\beta} &:=     \llim{\epsilon}{0}{\fracvarmin {f}{x}{n+\beta}   }   \epnt
	\end{align}
\end{definition}

\section{Minimal scale-dependent description}
\label{sec:minscale}

The notation $f(x)$ will be interpreted in two ways depending on the context: 
as the number $ f: x \rightarrow y$ if $f$ is evaluated  (a \textit{verb} sense) 
or as the prescription that  $f$ is used as an argument for the operator $P$ and the the image function is evaluated at the point $x$. 
Then $P : f(x) \rightarrow P[f] (x)$   (a \textit{noun} sense). 
The term operator will be used  as the map from one function to another.

\begin{definition}
	\label{def:embed}
	Define the real (resp. complex) scale-space  as the tensor products 
	\[
	\mathbb{S}  \cong \mathbb{R} \otimes \mathbb{R}_{+}, \ \ \mathbb{S}_C  \cong \mathbb{C} \otimes \mathbb{R}_{+} \epnt
	\]
	Let the function  $f: \mathbb{R} \mapsto \mathbb{C} $ be continuous on a compact subset of $\mathbb{R}$.
	Define the \textsl{scale embedding} of $f$ as the action of the operator
	\begin{flalign*}
	\mathcal{E} : & \ f(x )  \mapsto f (x, \epsilon) \cong f(x + \epsilon)  \\
	f (x, \epsilon): & \ \mathbb{S}  \mapsto  \mathbb{S}_C
	\end{flalign*}
	which adjoins a scale variable $\epsilon$ to the argument of the function.
\end{definition}
 
Unless stated otherwise we will further assume that $x$ is fixed and only the  scale variable $\epsilon$ is allowed to vary.

\subsection{Scale embedding of fractional velocities}
\label{sec:scfracvel}

As demonstrated previously, the fractional velocity has only "few" non-zero values. 
Therefore, it is useful to discriminate the set of arguments where the fractional velocity does not vanish.
\begin{definition}[Set of Change]
	The set of points where the fractional velocity exists finitely and $\fdiffpm{f}{x}{\beta}  \neq 0 $ will be denoted as the \textbf{set of change}
	$ \soc{\pm}{\beta} (f)$. 
\end{definition}  

The notion of scale invariance can be used complementary to scale dependence. 
Consider the fractional power function $x^\alpha$. For this function 
$
\fracvarplus{f}{0}{\beta}  = 1, \forall \epsilon >0
$.
Therefore, at this particular point the limit does not depend on the translation scale.
In a similar way, in general, the fractional variation of a slowly oscillating \holder{\alpha} function can be decomposed into a  
scale-invariant part  and a (possibly) scale-dependent one
\[
\fracvarplus{f}{x}{\alpha} = \fdiffplus{f}{x}{\alpha}   + \bigoh{1 }, \ x \in \soc{+}{\alpha} (f)
\]
Then we observe that the difference quotient
\[
\frac{ f (x+ \epsilon) - f(x) }{\epsilon}= \fdiffplus{f}{x}{\beta} \frac{1}{\epsilon^ {1-\alpha} } + \frac{\bigoh{ 1}}{\epsilon^ {1-\alpha}}
\]
is a diverging quantity, which is the staring point of the theory of Nottale.
What can be considered as a weakness of the original scale-relativistic approach is that numerical simulations of the  phenomena of interest are difficult to obtain since difference quotients diverge. 

One can define two types of \textbf{scale-dependent operators}  for a wide variety if  physical signals.
An extreme case of such signals are the singular functions, defined as
\begin{definition}
	\label{def:singular}
	A function $f(x)$ is called \underline{singular} on the interval $x \in[a,b]$, if it is i) non-constant, 
	ii) continuous;
	iii) $f^\prime(x)=0$ Lebesgue almost everywhere (i.e. the set of non-differentiability of $f$ is of measure 0) and 
	iv) $f(a) \neq f(b)$.
\end{definition}

Since for a singular signal the derivative either vanishes or it diverges then the rate of change for such signals cannot be characterized in terms of derivatives. 
One could apply to such signals  either the \textit{fractal variation operators}
of certain order or the difference quotients as Nottale does and avoid taking the limit. 
Alternately, as will be demonstrated further, the scale embedding approach can be used to reach the same goal. 
Singular functions can arise as point-wise limits of continuously differentiable ones.
Since the fractional velocity of a continuously-differentiable function vanishes we are lead to postpone taking the limit and only apply L'H\^ospital's rule, which under certain hypotheses detailed further will reach the same limit as $\epsilon \rightarrow 0$.
Therefore, we are set to introduce another pair of operators which in limit are equivalent to the fractional velocities 
notably these are the left (resp. right) \textit{scale velocity} operators:
\begin{flalign}
\svarmin {f}{x}{\beta} :=  \frac{1}{1 -\left\lbrace \beta\right\rbrace  } \epsilon^{\beta} 
\frac{\partial}{\partial \epsilon } f ( x- \epsilon) \ecma\\  
\svarplus {f}{x}{\beta} := \frac{1}{1 -\left\lbrace \beta\right\rbrace  } \epsilon^{\beta} 
\frac{\partial}{\partial \epsilon } f (x+ \epsilon)    \ecma
\end{flalign}
where $ \left\lbrace \cdot\right\rbrace  $ represents the fractional part taken for consistency with the case when $\beta=1$.
In this formulation  the value of $1-\beta$ can be considered as the magnitude of deviation from  linearity (or differentiability) of the signal at that point.
The $\epsilon$ parameter, which is not necessarily small, represents the scale of observation.

The equivalence in limit is guaranteed by the following result:
\begin{proposition}
	\label{prop:scaledif}
Let $f^\prime(x)$ be continuous and non-vanishing uniformly in $(x, x \pm \mu)$. 
Then
\[
 \lim\limits_{\epsilon \rightarrow 0 } \fracvarpm {f}{x}{1-\beta} =  \lim\limits_{\epsilon \rightarrow 0 } \svarpm {f}{x}{\beta}
\]
if one of the limits exits.	
\end{proposition}	
\begin{proof}
We require partial evaluation of the limit $\epsilon \rightarrow 0$.
The left expression transforms into the right expression by application of l'H\^opital's rule treating $x$ as a fixed parameter and $\epsilon$ as a variable \cite{Prodanov2015}. 
Consider the left scale velocity. That is  
\[
 \lim\limits_{\epsilon \rightarrow 0 } \frac{f(x+ \epsilon)- f(x)}{\epsilon^\beta} = 
 \lim\limits_{\epsilon \rightarrow 0 } \frac{f^\prime(x+ \epsilon) }{\beta \epsilon^{\beta-1}} =   
 \lim\limits_{\epsilon \rightarrow 0 } \svarpm {f}{x}{1-\beta}
\]
if $f^\prime(x+ \epsilon) \neq 0$. 
However, since $\epsilon$ is arbitrary then  $f^\prime(x ) \neq 0$ must hold uniformly in $(x, x \pm \mu)$.
The right case is proven in a similar manner. 
\end{proof}	

Therefore, we can identify the scale embedding of the derivative with its local generalization
in terms of \textit{fractional velocity}  as \svarpm {f}{x}{\beta} by the following definition 
 
\begin{definition}
	\label{def:derembed}
	Define the scale embeddings of the derivative $f^{\prime}(x ) \in \fclass{C}{}$ 
	as the  forward and backward $\beta$-scale velocities
	\begin{flalign*}
	\mathcal{E}^{\prime}_{+} :  \ f^{\prime}(x ) &  \mapsto  \svarplus {f}{x}{\beta}   \\
		\mathcal{E}^{\prime}_{-} :  \ f^{\prime}(x ) &  \mapsto  \svarmin {f}{x}{\beta} \\
		\svarpm {f}{x}{\beta} &:  \mathbb{C} \otimes \mathbb{R}_{+}  \mapsto \mathbb{C} \epnt
	\end{flalign*}
\end{definition}

\section{Link to fractional derivatives}
\label{sec:frderiv}

Consider the fractional derivative of $f$ in the sense of Caputo $ \frdiffix{\beta}{ a }   f (x)$ (for details see  \ref{sec:frint} ).
Then the following result can be stated:

\begin{theorem}
	Let $f(x) \in \fclass{C}{1}$ for $x$ uniformly in $[a, x)$.
	Then 
	\[
	\svarplus { \frdiffiix{\beta}{a} f}{x}{\alpha}  = 
	\frac{ \epsilon^{\alpha} \, {{\left( x+\epsilon-a \right) }^{\beta-1}}}{ \left(1 - \left\lbrace \alpha\right\rbrace \right) \Gamma\left( \beta\right) }  f \left( a\right) + 
	\dfrac{ \epsilon^{\alpha}}{1 - \left\lbrace \alpha\right\rbrace  } \frdiffix{1- \beta}{ a }  f (x + \epsilon) \epnt
	\]
\end{theorem}
\begin{proof}
We will compute the quantity
\[
\svarplus { \frdiffiix{\beta}{a} f}{x}{\alpha}  =
\dfrac{1}{1 - \left\lbrace \alpha\right\rbrace  } \dfrac{\epsilon^\alpha}{\Gamma(\beta)} 
\dfrac{\partial }{\partial \epsilon} \int_{a}^{x +\epsilon}   f \left( t \right)  \left( x-t +\epsilon \right)^{\beta -1}dt \epnt
\] 
Change of variables $z=x-t+\epsilon$ leads to
\[
\frdiffii{\beta}{x+\epsilon}{a} f (x + \epsilon) =\frac{1}{\Gamma\left( \beta\right) } \int_{0}^{x+\epsilon-a}f \left( x+\epsilon-z \right) \, {{z}^{\beta-1}}dz
\]
Since $\frdiffii{\beta}{x+\epsilon}{a} f (x + \epsilon)  \in AC ([a,b])$ then
differentiating leads to
\[
\frac{ f \left( a\right) \, {{\left( x+\epsilon-a \right) }^{\beta-1}}}{\Gamma\left( \beta\right) }
- \frac{1}{\Gamma\left( \beta\right) } \int_{x+\epsilon-a}^{0}\left( \frac{\partial}{\partial\,\epsilon}\, f\left( x+\epsilon -z \right) \right) \, {{z}^{\beta-1}}dz
\]
which after an additional change of variables leads to  
\[
\frac{ f \left( a\right) \, {{\left( x+\epsilon-a \right) }^{\beta-1}}}{\Gamma\left( \beta\right) }+ \frac{1}{\Gamma\left( \beta\right) }\int_{a}^{x +\epsilon}   f \left( t \right)  \left( x-t +\epsilon \right)^{\beta -1}dt
\]
The final expressions can be assembled into
\[
\svarplus { \frdiffiix{\beta}{a} f}{x}{\alpha}  =
\dfrac{ \epsilon^{\alpha}}{1 - \left\lbrace \alpha\right\rbrace  } \left( 
\frac{ f \left( a\right) \, {{\left( x+\epsilon-a \right) }^{\beta-1}}}{\Gamma\left( \beta\right) } +
 \frac{1}{\Gamma\left( \beta\right) }\int_{a}^{x +\epsilon}   f^{\prime} \left( t \right)  \left( x-t +\epsilon \right)^{\beta -1}dt
\right) 
\]
Therefore,
\[
\svarplus { \frdiffiix{\beta}{a} f}{x}{\alpha}  = 
  \frac{ \epsilon^{\alpha} \, {{\left( x+\epsilon-a \right) }^{\beta-1}}}{ \left(1 - \left\lbrace \alpha\right\rbrace \right) \Gamma\left( \beta\right) }  f \left( a\right) + 
  \dfrac{ \epsilon^{\alpha}}{1 - \left\lbrace \alpha\right\rbrace  } \frdiffix{1- \beta}{ a }  f (x +\epsilon)
\]
\end{proof}
Based on the theorem there are several useful corollaries to be stated:
\begin{corollary}
	If $f(a)$ is defined then under the same hypothesis for $f$
\[
\svarplus { \frdiffiix{\beta}{a} f - f(a)}{x}{\alpha}  =  
\dfrac{ \epsilon^{\alpha}}{1 - \left\lbrace \alpha\right\rbrace  } \frdiffix{1- \beta}{ a }  f (x  +\epsilon)
\]
\end{corollary}
\begin{corollary}
	If $f(a)$ is defined then under the same hypothesis for $f$
	\[
	\svarplus { \frdiffiix{1-\beta}{a} f - f(a)}{x}{\alpha}  =  
	\dfrac{ \epsilon^{\alpha}}{1 - \left\lbrace \alpha\right\rbrace  } \frdiffix{ \beta}{ a }  f (x +\epsilon )
	\]
\end{corollary}


\begin{proposition}
	If $f(a)$ is defined then under the same hypothesis for $f$
	\[
	\fdiffplus{ \frdiffiix{1-\beta}{a} f }{x}{\alpha} = \llim{\epsilon}{0}{}\
	\svarplus { \frdiffiix{1-\beta}{a} f - f(a)}{x}{1-\alpha}  =  
	\dfrac{1}{  \alpha   } \llim{\epsilon}{0}{} 
	\epsilon^{1-\alpha}\frdiffix{ \beta}{ a }  f (x +\epsilon)
	\]
\end{proposition}

\section{Characterization of the growth of some singular functions}
\label{sec:singular}

The \textbf{fractional velocity } is able to characterize growth of functions varying on fractal sets.
In the following example we demonstrate that the fractional velocity can be non-zero on a Cauchy-dense set.

\begin{example}
	De Rham's function is defined by the functional equations
	\[
	R_a(x):=
	\begin{cases} 
	0 , & x= 0 \\
	a  R_a(2 x) ,&  0  \leq x < \dfrac{1}{2}  \\
	(1-a)  R_a(2 x-1) +a ,&  \dfrac{1}{2} \leq x \leq 1 \\
	1, & x = 1
	\end{cases}  
	\]
	We will compute $\fdiffplus{R_a}{x}{\beta}$ in the interval $[0, 1]$.
		
	The computation follows from Kawamura \cite{Kawamura2011}.
	Consider the dyadic representation of the number $	x =  \overline{0.d_1 \ldots d_n}, \ d \in \{0, 1 \}$.
	Lomnicki and  Ulam \cite{Lomnicki1934} proved the arithmetic representation of $ R_a(x)$ :
	\begin{flalign*}
	R_a(x) &= \dfrac{a}{1-a} \sum\limits_{k=1}^{n} d_n a^{k-s_n} (1-a)^{s_n}  \ecma  
	s_n  = \sum\limits_{k=1}^{n} d_k 
	\end{flalign*}
	Then the increment of $R_a(x)$ for $\epsilon = 1/2^n $ is
	\[
	R_a  \left( x +  {1/2^n} \right)  - R_a \left( x \right)   = \dfrac{a}{1-a} a^{k-s_n} (1-a)^{s_n} 
	\]
	so that 
	\[
	\fracvarplus{R_a}{x}{\beta} =  \dfrac{a}{1-a} 2^{\beta k} a^{k-s_n} (1-a)^{s_n} =  
	{\left( 2^\beta a \right)}^{k} \left( \dfrac{1}{a} -1 \right)^{s_n-1} \epnt
	\]	
	Therefore, if $a= 1/2^\beta$ the ratio does not depend on $k$.
  
	Let's fix $k=n$ then if the representation is terminating (i.e. \textit{x} is dyadic rational)
	\[
	\fdiffplus{R_a}{x}{\beta} =  \left( 2^\beta - 1 \right)^{s_n-1} \epnt
	\]
	In contrast, if the representation is non-terminating then $\fdiffplus{R_a}{x}{\beta} = 0$ as $s_n$ grows without a bound.
\end{example}

\begin{figure}[h]
\begin{tabular}{ll}
A & B \\
\includegraphics[width=0.5\linewidth, bb=0 0 800 600]{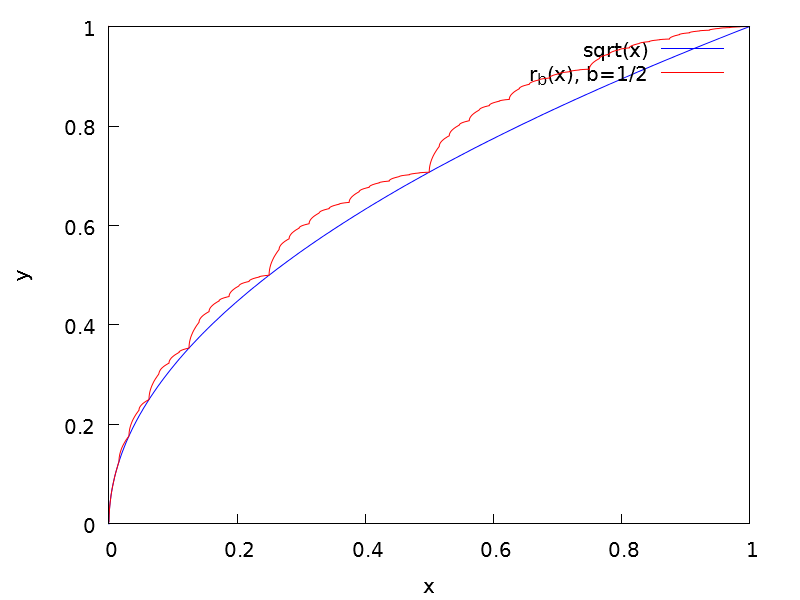} & 
\includegraphics[width=0.5\linewidth, bb=0 0 800 600]{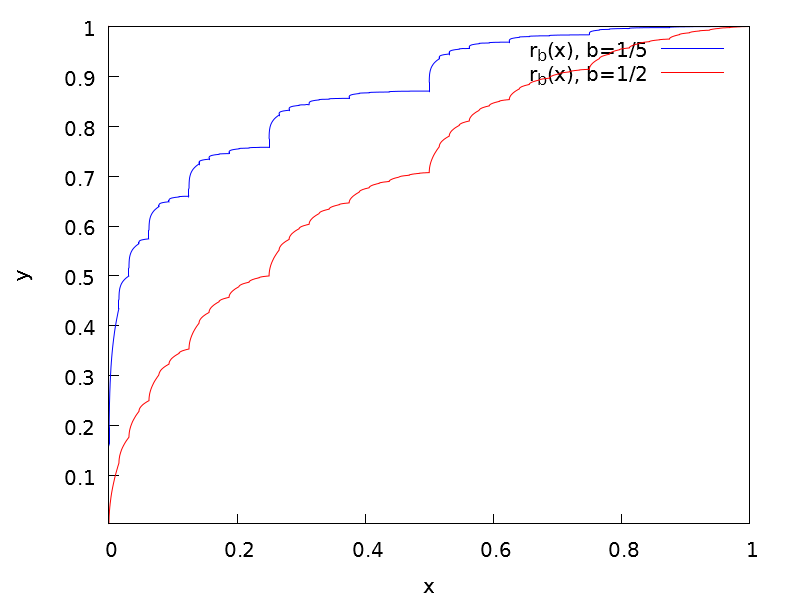}
\end{tabular}
\caption{De Rham's iteration function system}
A -- Dependence on the iteration at k=0 (solid line )and k=6 (dashed line). 
B -- Dependence on the shape parameter, b=$1/2$ (dashed line) and $b=1/5$ (solid line).
Plots are produced with the computer algebra system \textsl{Maxima}.
\label{fig:derham1s}
\end{figure}

This calculation can be summarized in the following proposition:
\begin{proposition}
	\label{prop:derham1}
	
	Let $\mathbb{Q}_2$ denote the set of dyadic rationals. 
	Let  $s_n = \sum\limits_{k=1}^{n} d_k $ denote the sum of the digits for the number  
	$x =  \overline{0.d_1 \ldots d_n}, \ d \in \{0, 1 \}$ in  binary representation, 
	then
	\[
	\fdiffplus{R_a}{x}{\beta}= \left\{  
	\begin{array}{ll} \left( 2^\beta -1 \right)^{s_{n}-1}, & x \in \mathbb{Q}_2 \\
	0, & x \notin  \mathbb{Q}_2
	\end{array}
	\right.
	\]
	for $\beta= - log_2 a$, $a \neq \dfrac{1}{2}$.
	For $\beta < - log_2 a$ by \cite{Prodanov2015} we have
	$\fdiffplus{R_a}{x}{\beta}=\fdiffmin{R_a}{x}{\beta}=0$.
	
\end{proposition}
\begin{remark}
The function given in the example is known under several different names -- "De Rham's function", 
"Lebesgue's singular function"  and also "Salem's singular function". 
This function  was also defined in different ways \cite{Cesaro1906,Faber1909}.
Lomnicki and Ulan \cite{Lomnicki1934}, for example, give a probabilistic construction as follows.
In a an imaginary experiment of flipping a possibly "unfair" coin with probability $a$ of heads (and 
$1 - a$ of tails). Then
$
R_a(x) = \mathbb{P} \left\lbrace t \leq x \right\rbrace 
$
after infinitely many trials  where $t$ is the record of the trials represented by a binary sequence.  
Finally, Salem \cite{Salem1943} gives a geometrical construction.
\end{remark}

We will reach the same result using the formalism of the scale embedding.
However, first we need to establish some general results.

\begin{lemma}[Scale recursion invariance]
	\label{th:screc}
	Consider the differentiable map  $
		\Phi: \fclass{F}{} \mapsto \fclass{F}{}
	$ acting on the complete metric space $\fclass{F}{}$ and the recursion 
	$ f_{n+1} (x)= \Phi [f_n] (x) $ generated by successive applications of $\Phi$.
	Suppose further that $f_0 \in \fclass{C}{1}$ about $x$.
	Then, considering the scale embedding,  there exists $\epsilon^\prime$, such that
	\[
	\svarpm{f_{n+1}}{x}{\alpha} = \svar{f_0}{x}{\alpha}{\epsilon^\prime \pm} \epnt
	\]
\end{lemma}
\begin{proof} 
	We push the recursive function to scale embedding. 
	By application of the scale  operator we get
	\[
	\svarplus{f_{n+1}}{x}{\alpha} = \frac{1}{1 - \{ \alpha \}} \epsilon^{\alpha} \, \frac{\partial \Phi}{\partial f_n}    \frac{\partial f_n}{\partial \epsilon} (x+\epsilon) = 
	\frac{\partial \Phi}{\partial f_n} \; \svarplus{f_{n}}{x}{\alpha} \epnt
	\]
	On the other hand, from the definition of recursion  we have
	$
	f_{n+1} (x)=  \prod_{k=1}^{n}  \Phi \ \circ f_{0}  \left( x \right) 
	$.
	Therefore, the first equation can be applied $n$ times to finally yield
	\begin{equation}
	\svarplus{f_{n+1}}{x}{\alpha} =  \left( \prod_{k=1}^{n} \frac{\partial \Phi}{\partial f_{n-k}} \right)
	\svarplus{f_0}{x}{\alpha} \epnt
	\end{equation}
	Then if the product is positive we can identify a positive number 
	\[
	\epsilon = \frac{1}{  \left( \prod_{k=1}^{n} \frac{\partial \Phi}{\partial f_{n-k}} \right)^{1/\alpha }} \epnt
	\]
	Similarly, if the product is negative we have a negative number
	\[
		\epsilon = - \frac{1}{  \left| \prod_{k=1}^{n} \frac{\partial \Phi}{\partial f_{n-k}}  \right| ^{1/\alpha }} \epnt
	\]
	The backward case is proven by identical reasoning. 
\end{proof}
Having established this result we can proceed to stating the main Theorem:
\begin{theorem}[Scale fixed point derivation]
	Consider the recursion $ f_{n+1} (x)= \Phi [f_n] (x) $ for a differentiable map $\Phi$ under the same hypothesis.
	Suppose that  $f_0 \in \fclass{C}{1}$  does not vanish in an open interval about \textit{x}.
	Then if 
	$
	\frac{\partial \, \Phi}{\partial f_{k}} > 1, \ \forall k >0
	$
	 there exists a null Cauchy sequence $ \{ \epsilon \}_k^{\infty}$, such that
	\[
	\llim{n}{\infty}{ \svarpm{f_{n}}{x}{1- \alpha}  } =\fdiffpm{f}{x}{  \alpha}  \epnt
	\]
	This sequence will be named \textbf{scale--regularizing} sequence.
\end{theorem}
\begin{proof}
To prove the claim we can observe that
	\[
	\epsilon_n = \frac{1}{  \left( \prod_{k=1}^{n} \frac{\partial \Phi}{\partial f_{n-k}} \right)^{1/\alpha }} 
	\]
	by the statement of the Scale recursion invariance lemma.
	Then since 	$
	\frac{\partial \, \Phi}{\partial f_{k}} > 1
	$ we have
	$ \epsilon_{n+1}< \epsilon_n <1 $. 
	Therefore, $ \{ p \, \epsilon_k\}_k$ is a Cauchy sequence for a positive  number \textit{p}.
	Further, setting $\beta =1- \alpha$ we proceed to evaluation of the limit
	\[
		\llim{n}{\infty}{ \svarpm{f_{n}}{x}{1-\beta}  } = 
		\llim{\epsilon}{0}{ \svarpm{f_{0}}{x}{1- \beta}  } = 
		\frac{1}{\beta}\llim{\epsilon}{0}{\epsilon^{1-\beta} f^{\prime} (x + \epsilon)}
	\]
	The last limit is also the limit of 
	$
	\llim{\epsilon}{0}{ \frac{  f^{\prime} (x + \epsilon) - f(x) }{ \epsilon^\beta}}
	$
	by the application of l'H\^opital's rule under the conditions identified in Prop. \ref{prop:scaledif}.
	The backward case can be proven by identical reasoning. 
\end{proof}

\begin{proposition}
Let the limit 
	$
	F(x) =\llim{n}{\infty}{f_{n}(x)}
	$
of the recursion $ f_{n+1} (x)= \Phi [f_n] (x) $
exists
and
	$
	\frac{\partial \, \Phi}{\partial f_{k}} > 1, \ \forall k >0 \epnt
	$
Then 
	\[
	\fdiffpm{f}{x}{  \alpha} = \llim{n}{\infty}{ \svarpm{f_{n}}{x}{1- \alpha}  } \epnt
	\]
\end{proposition}
These results can generalize to sequences of maps (i.e. iteration function systems) with disjoint domains.
In this case instead of a single recursion equation one needs to consider the Hutchinson operator, i.e. a system of equations  \cite{Hutchinson1981}.

\begin{figure}[h]
	\begin{center}
	\includegraphics[width=0.75\linewidth, bb=0 0 800 600]{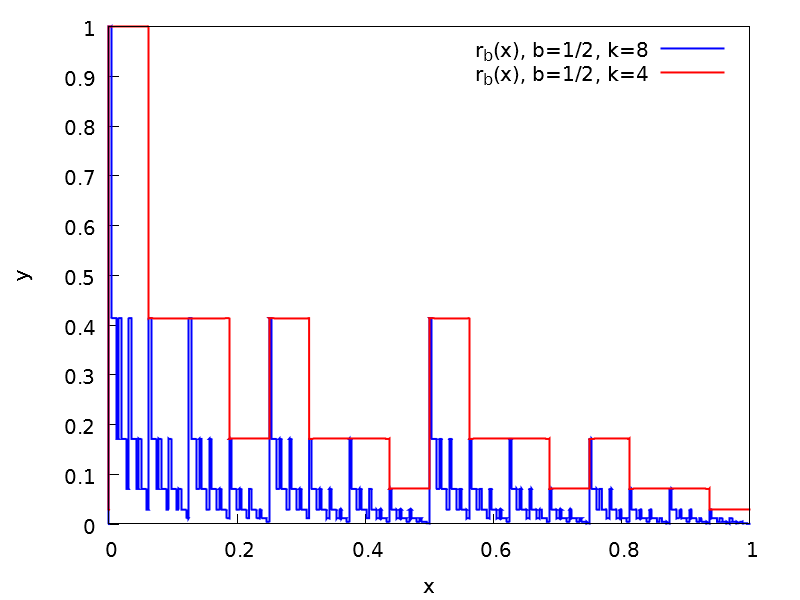} 
	\end{center}
	\caption{Fractional velocity approximation of De Rham's function system}
   Dependence on the iteration at k=8 (solid line) and k=4 (dashed line). 
	Plots are produced with the computer algebra system \textsl{Maxima}.
	\label{fig:derham1d}
\end{figure}
\begin{example}
	
To illustrate this proposition we will iterate the recursive function
$r_{0}(x, a)= x^{a}$, which verifies $ r_{0}(0, a)=0$ and  $ r_{0}(1, a)=1$. 
	\[
	r_n(x,a):= \left\{
	\begin{array}{ll}
	\frac{1}{2^a}  r_{n-1}(2 \, x, a) ,&  0  \leq x < \frac{1}{2}  \\
	(1-\frac{1}{2^a})  r_{n-1}(2 \, x-1, a) + \frac{1}{2^a} ,&  \frac{1}{2} \leq x \leq 1 \\
	\end{array} \right.
	\]
	The system is a slight re-parametrization of the original De Rham's function.

	By Banach's Fixed Point theorem and Hutchinson \cite{Hutchinson1981} we have
	$
	R_a(x) =\llim{n}{\infty}{r_{n}(x,a)}  
	$.
	
	We push the recursive function to scale embedding. 
	We observe that
	\[
	\svarplus {r_{0}}{x,a}{\alpha} = \frac{a }{1 - \{ \alpha \}} \epsilon^{\alpha} \left( x +\epsilon \right)^{a-1} \epnt
	\]
	Therefore, for $\alpha = 1- a $ we have	
	\[
	\svarplus {r_{0}}{x,a}{\alpha}=\frac{\epsilon^{\alpha} }{\left( x +\epsilon \right)^{\alpha}}
	\]
	so that \svarplus {r_{0}}{0,a}{\alpha}=1, which corresponds to scale invariance. 
	We observe that  scaling of the argument by 2  leads to repetition of the recursive pattern by a dyadic shift. 
	That is $r_{1} (0) \sim r_{0}(0)$    and  $  r_{1} (1/2) \sim r_{0}(0)$  and so on by induction.  
	So that if $x = \frac{m}{2^n} $ then $r_n (x) \sim r_{0}(0)$. 
	Further, we can discern two cases.
	
	Case 1, $x \leq 1/2 $ : Then application of the scale  operator leads to :
   \[ 
   \svarplus {r_{n}}{x,a}{\alpha} = \frac{2 }{1 - \{ \alpha \}} \left( \frac{\epsilon^{\alpha}}{2^a} \right)  \frac{\partial }{\partial \epsilon }r_{n-1}  (2 x+2 \epsilon, a) 
   \]
 Therefore, there is 
 \[
 \epsilon_n = \frac{1}{2^{n (1-a)/\alpha}} = \frac{1}{2^{n  }} 
 \]  
 Therefore, we can identify a scale-regularizing Cauchy sequence.
 
   Case 2, $x > 1/2$:  In a similar way :
   \[ 
   \svarplus {r_{n}}{x,a}{\alpha} = \frac{2 }{1 - \{ \alpha \}}  \left( \frac{2^a -1}{2^a}\right)   \epsilon^{\alpha} \frac{\partial }{\partial \epsilon } r_{n-1}(2 x + 2 \epsilon -1, a)
   \]
  Applying the same sequence results in a factor  $2^a - 1 \leq 1$. Therefore, the resulting transformation is a contraction. 
   Finally, we observe that if $ x = \overline{0.d_1 \ldots d_n} $ then the number of occurrences of Case 2 corresponds ot the amount of the digit $1$ in the binary notation, that is to 
   $s_n = \sum\limits_{k=1}^{n} d_k $.
   Therefore, we arrive to the same result as in Prop. \ref{prop:derham1}.
\end{example}

The former calculation can be formulated to algorithmically approximate the fractional velocity of the De Rham's function for any desired  iteration depth.
\begin{proposition}
$d_{0}(x, a)= 1$, 
\[
d_n(x,a):= 
\begin{cases} 
 d_{n-1}(2 \, x, a) ,&  0  \leq x < \frac{1}{2}  \\
\left(  2^a  -1 \right)   d_{n-1}(2 \, x-1, a) ,&  \frac{1}{2} \leq x \leq 1 \\
\end{cases} 
\]
\end{proposition}

\section{Discussion}
\label{sec:disc}

Singular functions, which  increase continuously yet have derivative zero almost everywhere, have fascinated mathematicians for over a century. 
De Rham's function has been studied in \cite{Rham1957} and recently by \cite{Berg2000,Kawamura2011,Bernstein2013}.
The method presented in the current work allows for simultaneous identification of point-wise H\"older exponents of singular functions together with their \textit{sets of change} from recursive functional equations. 
On a second place, a relationship between the local scaling operators and the non-local fractional differ-integrals has been demonstrated. 
The scale-space fixed point theorem demonstrates how scale iteration can give rise to power laws in the described phenomenon. 
The approach presented here is similar to the one used in the renormalization group equations \cite{Gell-Mann1954}. 
However, here it arises in the context of characterization of the  {set of change} of a singular function. 
Discussing further such connections is beyond the scope of the present work.

The approach presented in this paper is a generalization and formalization of the approach of Nottale. In my understanding the scale variable is not directly observable, therefore it must be treated in an abstract manner and each application should give an appropriate interpretation  in terms of the studied observable of interest.
The work demonstrates how an intuitive scale-relativistic approach can  be made mathematically rigorous, 
while retaining the intuition of the initial concept. 
It is demonstrated that abandoning the assumption of linearity allows one to reconcile the mathematics of the limiting process with the intuitive treatment of scale as a variable. 
In addition correspondence between the scale-relativistic approach to more traditional mathematical frameworks, such as fractional calculus can be also demonstrated.

\section*{Acknowledgments}
The work has been supported in part by a grant from Research Fund - Flanders (FWO), contract numbers  G.0C75.13N, VS.097.16N.

\appendix

 
\section{Holder functions}
\label{sec:holder}

\begin{definition}
	\label{def:holder}
	Consider the (extended) mapping $f: \mathbb{R} \rightarrow \mathbb{C}  $.
	We say that $f$ is of (point-wise) H\"older  class \holder{\beta} of order $\beta$ ( $0< \beta \leq 1$ ) if for a given $\mathbf{x}$ there exist two positive constants 
	$C, \delta \in \mathbb{R} $ that for an arbitrary $ y$ in the domain of $f$ and given $|x-y| \leq \delta$ fulfill the inequality
	$
	| f (x) - f (y) |  \leq C |x-y|^\beta
	$, where $| \cdot |$ denotes the norm of the argument.  	
	
	For mixed orders $n +\beta $ the  H\"older class \holder{n+ \beta}  designates the functions  for which the inequality 
	\[
	| f (x)  - P_n (x-y) |  \leq C |x-y|^{n +\beta}  
	\]
	holds under the same hypothesis for $C, \delta $ and \textit{y}. $P_n (z)$ designates the Taylor polynomial  
	$
	P_n (z) = f (y)+ \sum\limits_{k=1}^{n}{ a_k z^k} 
	$ of order $n \in \mathbb{N}$.
\end{definition}

\section{Fractional integrals and derivatives }
\label{sec:frint}
The Riemann-Liouville fractional integral (or differ-integral)  defines a weighted average  of the function over an interval 
with a varying endpoint using a power law  weighting function.
The Riemann-Liouville differ-integral of order $\beta \geq 0$ is defined \cite{Oldham1974} as
\[
\frdiffiix{\beta}{ a }  f (x) = 
\dfrac{1}{\Gamma(\beta)} \int_{a}^{x}   f \left( t \right)  \left( x-t \right)^{\beta -1}dt 
\]
where $\Gamma(x) $ is the Euler's Gamma function. 
The fractional derivative of a function in the sense of Riemann-Liouville is defined in terms of the fractional integral as 
\begin{equation}
\frdiffixr{n +\beta}{ a }   f (x) = \left( \frac{d}{dx} \right)^n \frdiffiix{n-\beta}{ a }  f (x)
\end{equation}
which is usually specialized for $n=1$   in an explicit form by
\[
\frdiffixr{\beta}{ a }  f (x) = \dfrac{1}{\Gamma(1- \beta)}  \frac{d}{dx}  \int_{a}^{x}\frac{  f \left( t\right) }{{\left( x-t\right) }^{\beta }}dt \epnt
\]
 
The left Riemann-Liouville derivative of a function $f$ is defined for members of the functional space 
defined as
\[
E^{\alpha}_{a,+} ([a, b]):= \left\lbrace 
f \in L_1 , \frdiffiix{1- \alpha}{ a }  f (x) \in AC([a,b]) \right\rbrace 
\]
in Samko et al. \cite{Samko1993}[Definition 2.4, p. 44].

The fractional derivative of a function in the sense of Caputo  is defined as 
\begin{equation}
\frdiffix{\beta}{ a }  f (x) =  \frdiffiix{n-\beta}{ a }  f^{(n)} (x) \epnt
\end{equation}
or equivalently  for $n=1$  \cite{Caputo1967,Caputo1971}
\[
\frdiffix{\beta}{ a }  f (x) = \dfrac{1}{\Gamma(1- \beta)}\int_{a}^{x}\frac{  f^{\prime}\left( t\right) }{{\left( x-t\right) }^{\beta}}dt \epnt
\] 
The operator is defined for functions of  \fclass{C}{1}(a,b).
Both definitions coincide for problems where the function and its first $n$ derivatives vanish at the lower limit of integration, i.e. when $f(a)=0, \ldots , f^{(n)} (a) =0 $.
Caputo's definition is better suited for problems where the function and its derivatives do not vanish at the boundary of the domain, because in this case it is given as kind of regularization of the Riemann-Liouville differ-integral  to avoid divergences.  
Caputo's derivative is a left inverse of the fractional integral:
\begin{equation}
\label{eq:inverse1}
\frdiffix{\beta}{ a } \circ \frdiffiix{\beta}{ a }    f = f(x) \ecma
\end{equation}
while the  fractional integral is a conditional inverse of Caputo's derivative:  
\begin{equation}
\label{eq:inverse2}
\frdiffiix{\beta}{ a } \circ  \frdiffix{\beta}{ a }  f = f(x) - f(a^+)
\end{equation}
in the case when $\frdiffix{\beta}{ a }  f \neq 0 $.

\section*{Bibliography}

\end{document}